\documentclass[11pt]{article}
\usepackage{amsmath,fullpage,amssymb,amsthm,enumerate}
\usepackage{tikz,color}
\usepackage{hyperref}

\newtheorem{theorem}{Theorem}[section]
\newtheorem{proposition}[theorem]{Proposition}

\newtheorem{corollary}[theorem]{Corollary}

\theoremstyle{definition}

\newtheorem{example}[theorem]{Example}

\def\B{\mathcal B}
\def\DP{\mathcal D}

\def\bij{\phi}

\DeclareMathOperator\area{area}
\DeclareMathOperator\pk{pk}
\DeclareMathOperator\val{val}
\renewcommand\sp{\operatorname{sp}}
\renewcommand\sl{\operatorname{sl}}
\DeclareMathOperator\ob{ob}
\DeclareMathOperator\eb{eb}
\DeclareMathOperator\height{height}
\newcommand\iniU{\operatorname{ini}_U}
\newcommand\finD{\operatorname{fin}_D}
\newcommand\iniu{\operatorname{ini}_u}
\newcommand\find{\operatorname{fin}_d}
\DeclareMathOperator\ret{ret}
\DeclareMathOperator\shp{shp}
\DeclareMathOperator\shv{shv}
\newcommand{\uph}{{\uparrow}}
\newcommand{\downh}{{\downarrow}}

\def\U{-- ++(1,1) circle(1.2pt)}
\def\D{-- ++(1,-1) circle(1.2pt)}
\def\H{-- ++(1,0) circle(1.2pt)}

\def\N{-- ++(0,1) circle(1.2pt)}
\def\S{-- ++(0,-1) circle(1.2pt)}

\title{A bijection between bargraphs and Dyck paths}

\author{Emeric Deutsch\thanks{NYU Tandon School of Engineering, Brooklyn, NY 11201, USA.} \and Sergi Elizalde\thanks{Department of Mathematics, Dartmouth College, Hanover, NH 03755, USA. {\tt sergi.elizalde@dartmouth.edu}. Partially supported by grant \#280575 from the Simons Foundation.}
}

\date{}

\begin{document}

\maketitle

\begin{abstract}
Bargraphs are a special class of convex polyominoes. They can be identified with lattice paths with unit steps north, east, and south that start at the origin, end on the $x$-axis, and stay strictly above the $x$-axis everywhere except at the endpoints. 
Bargraphs, which are used to represent histograms and to model polymers in statistical physics, have been enumerated in the literature by semiperimeter and by several other statistics, using different methods such as the wasp-waist decomposition of Bousquet-M\'elou and Rechnitzer,
and a bijection with certain Motzkin paths.

In this paper we describe an unusual bijection between bargraphs and Dyck paths, and study how some statistics are mapped by the bijection. 
As a consequence, we obtain a new interpretation of Catalan numbers, as counting bargraphs where the semiperimeter minus the number of peaks is fixed.
\end{abstract}



\section{Introduction}

A bargraph is a lattice path with steps $U=(0,1)$, $H=(1,0)$ and $D=(0,-1)$ that starts at the origin and ends on the $x$-axis, stays strictly above the $x$-axis everywhere except at the endpoints, and has no pair of consecutive steps of the form $UD$ or $DU$. Sometimes it is convenient to identify a bargraph with the corresponding word on the alphabet $\{U,D,H\}$, and sometimes with the sequence of heights (i.e., $y$-coordinates) of the $H$ steps of the path, thus interpreting the bargraph as a composition.

Bargraphs appear in the literature with different names. 
As a special class of convex polyominoes, they have been studied and enumerated by 
Bousquet-M\'elou and Rechnitzer~\cite{BMR}, Prellberg and Brak~\cite{PB}, and Fereti\'c~\cite{Fer}.
The distribution of several bargraph statistics has been given in a series of papers by 
Blecher, Brennan, Knopfmacher and Prodinger~\cite{BBK_levels,BBK_peaks,BBK_parameters,BBK_walls,BBKP}. All of the above papers rely on a decomposition of bargraphs that is used to obtain equations satisfied by the corresponding generating functions.
Recently, the enumeration of bargraphs with respect to some additional statistics has been obtained in~\cite{DE} by using a simple bijection between bargraphs and Motzkin paths without peaks or valleys, together with the recursive structure of Motzkin paths.

The {\em semiperimeter} of a bargraph $B$ is defined as the number of $U$ steps plus the number of $H$ steps, and denoted by $\sp(B)$. It coincides with the semiperimeter of the closed polyomino determined by the bargraph and the $x$-axis.
Let $\B$ denote the set of bargraphs, and let $\B_n$ denote those of semiperimeter $n$. 
An important bargraph statistic in this paper is the {\em number of peaks}, which are occurrences of $UH^\ell D$ for some $\ell\ge1$.

A Dyck path is a lattice path with steps $u=(1,1)$ and $d=(1,-1)$ that starts at the origin, ends on the $x$-axis, and stays weakly above the $x$-axis. Let $\DP$ denote the set of Dyck paths, and let $\DP_n$ denote those of semilength $n$, where the semilength of a path $P$ is defined as half of the number of steps, and denoted by $\sl(P)$. 
It is well known that $|\DP_m|=C_m=\frac{1}{m+1}\binom{2m}{m}$, the $m$th Catalan number.

In Section~\ref{sec:bij} we introduce a bijection between bargraphs and Dyck paths. An unusual feature of our bijection is that it does not preserve ``size'' under the natural definitions, such as semiperimeter or semilength. Instead, the statistics preserved by the bijection are more subtle, as discussed in Section~\ref{sec:properties}. As a consequence of the bijection, we obtain a new interpretation of the Catalan numbers in terms of bargraphs. An alternative proof of this result, using generating functions, is given in Section~\ref{sec:gf}.

\section{The bijection}\label{sec:bij}

In this section we define a bijection $\bij$ from Dyck paths to bargraphs, with the property that the semilength of a Dyck path becomes the semiperimeter minus the number of peaks of the corresponding bargraph.

Given a Dyck path $P$, define the height of each step of $P$ as the $y$-coordinate of its highest point. Consider the sequence of heights of the steps of $P$. To construct the bargraph $\bij(P)$ from this sequence, turn each maximal block of $c$ consecutive letters $h$ into $\lfloor\frac{c}{2}\rfloor$ columns of height~$h$. 
 
For example, the Dyck path in Figure~\ref{fig:bij} has height sequence ${\color{blue}11111}{\color{red}2}{\color{green}3333}{\color{red}2222}{\color{green}3333}{\color{red}222}{\color{blue}1111}{\color{red}22}{\color{blue}1}$, and thus its corresponding bargraph has column heights ${\color{blue}11}{\color{green}33}{\color{red}22}{\color{green}33}{\color{red}2}{\color{blue}11}{\color{red}2}$.

\begin{figure}[htb]
\centering
    \begin{tikzpicture}[scale=0.55]
     \draw[blue](0,6) circle (1.2pt)  \U\D\U\D\U coordinate(last);  
     \draw[red](last) \U coordinate(last);  
     \draw[green](last) \U\D\U\D coordinate(last);  
     \draw[red](last) \D\U\D\U coordinate(last);  
     \draw[green](last) \U\D\U\D coordinate(last);  
     \draw[red](last)\D\U\D coordinate(last);  
     \draw[blue](last) \D\U\D\U coordinate(last);  
     \draw[red](last)\U\D coordinate(last);  
     \draw[blue](last)\D;  
     \draw[thin,dotted](0,6)--(28,6);
     \draw(14,4.7) node {$\downarrow$};
     \draw(8,0) circle(1.2pt)  \N coordinate(last);  
     \draw[blue](last) \H\H coordinate(last);  
     \draw(last)\N\N coordinate(last);  
     \draw[green](last) \H\H coordinate(last);  
     \draw(last)\S coordinate(last);  
     \draw[red](last)\H\H coordinate(last);  
     \draw(last) \N coordinate(last);  
     \draw[green](last) \H\H coordinate(last);  
     \draw(last) \S coordinate(last);  
     \draw[red](last) \H coordinate(last);  
     \draw(last) \S coordinate(last);  
     \draw[blue](last) \H\H coordinate(last);  
     \draw(last) \N coordinate(last);  
     \draw[red](last) \H coordinate(last);  
     \draw(last) \S\S; 
     \draw[thin,dotted](8,0)--(20,0);
    \end{tikzpicture}
   \caption{A Dyck path and its corresponding bargraph.}\label{fig:bij}
\end{figure}
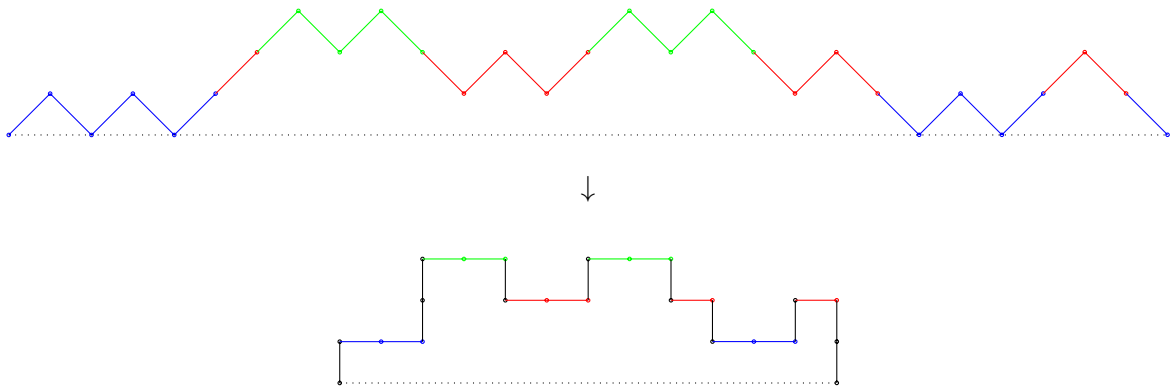

\medskip

Next we describe the inverse. Consider a bargraph described by a sequence of column heights $h_1^{a_1}h_2^{a_2}\dots h_r^{a_r}$.
By inserting terms with $a_j=0$ if necessary, we can assume that $h_1=h_r=1$ and $|h_{i+1}-h_i|=1$ for all $i$. Define $h_0=h_{r+1}=0$ by convention. To construct a Dyck path, for $i=1,2,\dots,r$, we replace $h_i^{a_i}$ with
$$\begin{cases}  
(ud)^{a_i} & \mbox{if $h_{i-1}<h_i>h_{i+1}$}, \\
(ud)^{a_i}u & \mbox{if $h_{i-1}<h_i<h_{i+1}$}, \\
(du)^{a_i}d & \mbox{if $h_{i-1}>h_i>h_{i+1}$},\\
(du)^{a_i} & \mbox{if $h_{i-1}>h_i<h_{i+1}$}.
\end{cases}$$

For example, the bargraph with column heights ${\color{blue}11}{\color{green}33}{\color{red}22}{\color{green}33}{\color{red}2}{\color{blue}11}{\color{red}2}$ is first written as ${\color{blue}1^2} {\color{red}2^0} {\color{green}3^2} {\color{red}2^2} {\color{green}3^2} {\color{red}2^1} {\color{blue}1^2} {\color{red}2^1} {\color{blue}1^0}$, which then becomes
${\color{blue}(ud)^2u}{\color{red}u}{\color{green}(ud)^2}{\color{red}(du)^2}{\color{green}(ud)^2}{\color{red}(du)d}{\color{blue}(du)^2}{\color{red}(ud)}{\color{blue}d}$.

\section{Properties of $\bij$}\label{sec:properties}

Next we show how some statistics on Dyck paths are mapped by the bijection $\bij$ into statistics on  bargraphs. 
We start by defining some statistics on bargraphs and Dyck paths.

For a bargraph $B$, denote by $\#H(B)$, $\#U(B)$ and $\#D(B)$ the number of $H$, $U$ and $D$ steps in $B$, respectively. 
Note that, by definition, $\sp(B)=\#H(B)+\#U(B)$. The {\em area} of the region under a bargraph $B$ and above the $x$-axis is denoted by $\area(B)$.

A {\em peak} (resp. {\em valley}) in a bargraph is an occurrence of $UH^\ell D$ (resp. $DH^\ell U$) for some $\ell\ge1$. Denote by $\pk(B)$ the number of peaks of $B$, and by $\val(B)$ its number of valleys. Note that $\pk(B)=\val(B)+1$, since peaks and valleys in a bargraph alternate, always starting and ending with a peak. Let $\shv(B)$ be the sum of the heights of the valleys of $B$, where the height of a valley is the $y$-coordinate of its $H$ steps.
Denote by $\#H_1(B)$ the number of $H$ steps of $B$ at height $1$ (equivalently, number of columns of $B$ of height $1$).

A {\em peak} in a Dyck path is an occurrence of $ud$, and a {\em return} is a $d$ step that ends on the $x$-axis. For $P\in\DP$, let $\Lambda(P)$ denote its number of peaks, and $\ret(P)$ its number of returns. Defining the height of a peak as the $y$-coordinate of its highest point, let $\shp(P)$ denote the sum of the heights of the peaks of $P$.

Define the height of a Dyck path or a bargraph to be the maximum $y$-coordinate that it reaches. Denote by $\iniu(P)$ (resp. $\find(P)$) the number of initial $u$ steps (resp. final $d$ steps) in a Dyck path $P$, and define $\iniU(B)$ and $\finD(B)$ on a bargraph $B$ similarly.

Table~\ref{tab:stat} summarizes the notation that we use for statistics on bargraphs and Dyck paths.

\begin{table}[htb]
\centering
\begin{tabular}{c|c||c|c}
\multicolumn{2}{c||}{Statistics on bargraphs} & \multicolumn{2}{c}{Statistics on Dyck paths} \\
\hline \hline
$\sp$ & semiperimeter & $\sl$ & semilength \\
\hline
$\pk$ & number of peaks & $\Lambda$ & number of peaks \\
\hline
$\val$ & number of valleys & $\ret$ & number of returns \\
\hline
$\shv$ & sum of heights of valleys & $\shp$ & sum of heights of peaks \\
\hline
$\iniU$ & number of initial $U$ steps & $\iniu$ & number of initial $u$ steps \\
\hline
$\finD$ & number of final $D$ steps & $\find$ & number of final $d$ steps \\
\hline
$\#H,\#U,\#D$ & number of $H,U,D$ steps & $\eb$ & number of even blocks in height sequence \\
\hline 
$\#H_1$ & number of $H$ steps at height $1$ & $\ob$ & number of odd blocks in height sequence

\end{tabular}
\caption{Notation for statistics on bargraphs and Dyck paths}
\label{tab:stat}
\end{table}

\begin{example}
The bargraph $B$ in Figure~\ref{fig:bij} has $\sp(B)=17$, $\#H(B)=12$, $\#U(B)=\#D(B)=5$, $\area(B)=24$, $\pk(B)=3$, $\shv(B)=3$, $\#H_1(B)=4$, $\height(B)=3$, $\iniU(B)=1$, $\finD(B)=2$.
\end{example}

\begin{theorem}\label{thm:bijstat}
Let $P\in\DP$ and let $B=\bij(P)\in\B$. Then 
\begin{enumerate}[(a)]
\item $\sl(P)=\sp(B)-\pk(B)$,
\item $\Lambda(P)=\#H(B)-\val(B)$,
\item $\shp(P)=\area(B)-\shv(B)$,
\item $\height(P)=\height(B)$,
\item $\iniu(P)=\iniU(B)$,
\item $\find(P)=\finD(B)$.
\item $\ret(P)=\#H_1(B)+1$ (unless $P$ and $B$ have height $1$, in which case $\ret(P)=\#H_1(B)$).
\end{enumerate}
\end{theorem}

\begin{proof}
Let us first prove part (a).
The height sequence of $P$ can be decomposed into maximal blocks of steps of the same height. Let $\eb(P)$ be the number of even blocks, and let $\ob(P)$ be the number of odd blocks. In the construction of $B=\bij(P)$, each block of size $c$ produces $\lfloor\frac{c}{2}\rfloor$ horizontal steps in the bargraph. Thus, the number of $H$ steps in $B$ equals $\#H(B)=\sl(P)-\frac{\ob(P)}{2}$.

Additionally, the height changes by one from each block to the next, which translates into a $U$ or a $D$ in the bargraph. Thus $\#U(B)+\#D(B)=\eb(P)+\ob(P)+1$. More specifically, the two consecutive height changes occurring at the beginning and at the end of a given block (with the convention that the height is $0$ outside the path) have the same sign if and only if the block has odd length.

For example, in the path of Figure~\ref{fig:bij}, denoting by $\uph$ and $\downh$ height changes of $+1$ and $-1$ respectively, we get
$\uph{\color{blue}11111}\uph{\color{red}2}\uph{\color{green}3333}\downh{\color{red}2222}\uph{\color{green}3333}\downh{\color{red}222}\downh{\color{blue}1111}\uph{\color{red}22}\downh{\color{blue}1}\downh$. In particular, the number of even blocks equals the number of peaks plus the number of valleys (which in turn equals the number of peaks minus one), and so $\eb(P)=2\pk(B)-1$.

Putting the above observations together,
\begin{align*}\sp(B)&=\#H(B)+\frac{\#U(B)+\#D(B)}{2}=\sl(P)-\frac{\ob(P)}{2}+\frac{\eb(P)+\ob(P)+1}{2}\\
&=\sl(P)+\frac{\eb(P)+1}{2}=\sl(P)+\pk(B),
\end{align*}
as desired.

To prove parts (b) and (c), note that in the above decomposition of $P$ into maximal blocks of steps of the same height, all peaks of $P$ occur within the blocks, since the two steps in a peak $UD$ have the same height. Additionally, the height of such a peak coincides with the height of the block containing it, and thus with the height of the corresponding columns in the bargraph. Each odd block in $P$ of size $2i+1$ contains $i$ peaks. Similarly, each even block of size $2i$ contains $i$ peaks, unless the two blocks surrounding it are higher than the block itself, in which case it contains $i-1$ peaks. Blocks surrounded by higher blocks correspond to valleys of $B$. It follows that
$\Lambda(P)=\sl(P)-\frac{\ob(P)}{2}-\val(B)=\#H(B)-\val(B)$. Additionally, since the area of a bargraph is the sum of the heights of its columns, we get $\shp(P)=\area(B)-\shv(B)$.

Parts (d), (e) and (f) follow immediately from the construction of $\bij$.

Finally, to prove part (g), note that all returns of $P$ occur in blocks of height $1$. If $P$ (equivalently, $B$, by part~(c)) has height $1$, then it is clear that $\ret(P)=\sl(P)=\#H_1(B)$. If $P$ has height at least $2$, then the blocks of height $1$ are the first and last block, which are necessarily odd, plus possibly some even blocks. If the first block has length $2i+1$ (for some $i$), it contributes $i$ returns to $P$ and $i$ $H$ steps at height $1$ to $B$. If the last block has length $2i+1$, it contributes $i+1$ returns to $P$ and $i$ $H$ steps at height $1$ to $B$.
Any even block at height $1$ of length $2i$ contributes $i$ returns to $P$ and $i$ $H$ steps at height $1$ to $B$. The equality $\ret(P)=\#H_1(B)+1$ follows.
\end{proof}

The following is an immediate consequence of Theorem~\ref{thm:bijstat}(a), giving a new interpretation of the Catalan numbers.

\begin{corollary}\label{cor:Catalan}
$$|\{B\in\B: \sp(B) - \pk(B) = m\}|=C_m.$$ 
\end{corollary}

\begin{example}
For  $m = 1, 2, 3$, the set $\{B\in\B: \sp(B) - \pk(B) = m\}$ consists of all bargraphs 
of semiperimeter $2, 3, 4$, respectively, since each of these bargraphs has $1$ peak. For $m = 4$, this set, drawn in Figure~\ref{fig:m4}, consists of all $13$ bargraphs of semiperimeter 5  (each of them having $1$ peak) and the only bargraph of semiperimeter 6 that has 2 peaks.
\end{example}

\begin{figure}[htb]
\centering
    \begin{tikzpicture}[scale=0.55]
    \draw(0,0) circle(1.2pt)  \N\H\H\H\H\S; \draw[thin,dotted](0,0)--(4,0);
    \draw(5,0) circle(1.2pt)  \N\H\H\N\H\S\S; \draw[thin,dotted](5,0)--(8,0);
    \draw(9,0) circle(1.2pt)  \N\H\N\H\S\H\S; \draw[thin,dotted](9,0)--(12,0);
    \draw(13,0) circle(1.2pt)  \N\N\H\S\H\H\S; \draw[thin,dotted](13,0)--(16,0);
    \draw(17,0) circle(1.2pt)  \N\H\N\H\H\S\S; \draw[thin,dotted](17,0)--(20,0);
    \draw(21,0) circle(1.2pt)  \N\N\H\H\S\H\S; \draw[thin,dotted](21,0)--(24,0);
    \draw(25,0) circle(1.2pt)  \N\N\H\H\H\S\S; \draw[thin,dotted](25,0)--(28,0);
    \draw(25,0) circle(1.2pt)  \N\N\H\H\H\S\S; \draw[thin,dotted](25,0)--(28,0);
\begin{scope}[shift={(1,-5)}]
    \draw(0,0) circle(1.2pt)  \N\H\N\N\H\S\S\S; \draw[thin,dotted](0,0)--(2,0);
    \draw(4,0) circle(1.2pt)  \N\N\N\H\S\S\H\S; \draw[thin,dotted](4,0)--(6,0);
    \draw(8,0) circle(1.2pt)  \N\N\H\N\H\S\S\S; \draw[thin,dotted](8,0)--(10,0);
    \draw(12,0) circle(1.2pt)  \N\N\N\H\S\H\S\S; \draw[thin,dotted](12,0)--(14,0);
    \draw(16,0) circle(1.2pt)  \N\N\N\H\H\S\S\S; \draw[thin,dotted](16,0)--(18,0);
    \draw(20,0) circle(1.2pt)  \N\N\N\N\H\S\S\S\S; \draw[thin,dotted](20,0)--(21,0);
     \draw(23,0) circle(1.2pt)  \N\N\H\S\H\N\H\S\S;  \draw[thin,dotted](23,0)--(26,0);
\end{scope}
    \end{tikzpicture}
   \caption{The 14 bargraphs with $\sp(B)-\pk(B)=4$.}\label{fig:m4}
\end{figure}
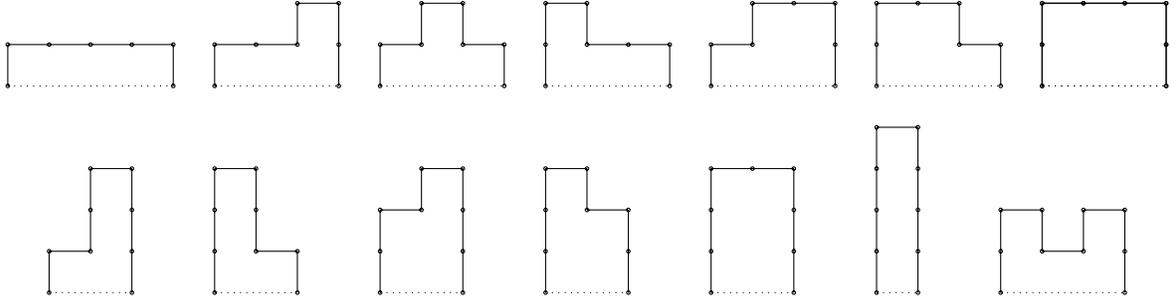

Table~\ref{tab:peaks} displays the number of bargraphs with a given semiperimeter (up to $12$) and a given number of peaks. Corollary~\ref{cor:Catalan} states that the diagonal sums of this table (which is the triangular sequence \cite[A271940]{OEIS}) are the Catalan numbers.
This observation for small values of $n$ led to the results in this paper.

\begin{table}[htb]
$$
\begin{array}{r|rrrr}
n\backslash k & 1 & 2 & 3 & 4\\ \hline
2 & 1 &&& \\ 
3 & 2 &&& \\ 
4 & 5 &&& \\ 
5 & 13 &&& \\ 
6 & 34 & 1 && \\ 
7 & 89 & 8 && \\ 
8 & 233 &42&& \\ 
9 & 610 &183 & 1 & \\ 
10 & 1597 &717 & 13 & \\ 
11 & 4181 &2622 & 102 & \\ 
12 & 10946 & 9134 & 624 & 1 \\ 
\end{array}
$$
\caption{Number of bargraphs with semiperimeter $n$ having $k$ peaks.
The diagonal sums of this table, which count bargraphs where  $n-k$ is fixed, are the Catalan numbers $1, 2, 5, 14, 42, 132, 429, \dots$.
}
\label{tab:peaks}
\end{table}

\medskip

We end this section by describing other properties of $\bij$, which can be used to give a recursive description of this map. Interpreting a bargraph $B$ as a lattice path, the path $UBD$ is the bargraph whose height sequence is obtained by adding one to each entry of the height sequence of $B$. On the other hand, given bargraphs $B$ and $B'$, the expression $B\circ B'$ denotes the concatenation of $B$ and $B'$ viewed as height sequences. 
Let  $\mathbf{1}$ denote the bargraph consisting of one column of height $1$.

\begin{proposition}
Let $P,P'\in\DP$, and let $B=\bij(P)$, $B'=\bij(P')$. Then 
\begin{enumerate}[(i)]
\item $\bij(uPd)=UBD$,
\item $\bij(PP')=\begin{cases} B\circ \mathbf{1} \circ B' & \text{if $\height(P)\ge2$ and $\height(P')\ge2$,}\\
 B\circ B' & \text{otherwise.}\end{cases}$
\end{enumerate}
\end{proposition}

\begin{proof}
To prove part (i), note that elevating the path $P$ increases the entries of its height sequence by one, and adds a $1$ at the beginning and a $1$ at the end of it, neither of which contributes new columns to the bargraph.

Part (ii) follows from the observation that, if $P\in\DP$ has $\height(P)\ge2$, then the initial and final maximal blocks of $1$'s in the height sequence of $P$ have odd sizes.
\end{proof}

\section{Generating functions}\label{sec:gf}

In this subsection we give an alternative proof of Corollary~\ref{cor:Catalan} using generating functions. Let $$G(t,z)=\sum_{B\in\B} t^{\pk(B)} z^{\sp(B)}$$ be the generating function for bargraphs according to the number of peaks and the semiperimeter.
It was shown in~\cite{BBK_peaks} (and can be easily derived from~\cite{DE}) that $G$ satisfies
\begin{equation}\label{eq:G} z(1-z)G(t,z)^2-(1-3z+z^2+tz^3)G(t,z)+tz^2(1-z)=0.\end{equation}
Let $$H(t,z)=G(1/t,tz)=\sum_{B\in\B} t^{\sp(B)-\pk(B)} z^{\sp(B)}$$ be the generating function where $t$ keeps track of the semiperimeter minus the number of peaks. Then, substituting into~\eqref{eq:G}, we see that $H$ satisfies
\begin{equation}\label{eq:H} tz(1-tz)H(t,z)^2-(1-3tz+t^2z^2+t^2z^3)H(t,z)+tz^2(1-tz)=0.\end{equation}
Setting $z=1$ and simplifying, we get 
$$tH(t,1)^2-(1-2t)H(t,1)+t=0,$$
yielding $$H(t,1)=\sum_{B\in\B} t^{\sp(B)-\pk(B)} = \frac{1-\sqrt{1-4t}}{2t}-1=\sum_{m\ge1} C_m t^m.$$

\end{document}